\documentclass{amsart}
\usepackage{graphicx}

\newtheorem{theorem}{Theorem}[section]
\newtheorem{lemma}[theorem]{Lemma}
\newtheorem{prop}[theorem]{Proposition}
\newtheorem{cor}[theorem]{Corollary}

\newtheorem{rmk}[theorem]{Remark}
\newtheorem{que}[theorem]{Question}
\newtheorem{defn}[theorem]{Definition}

\newcommand{\Gab}{G  -  a,b}
\newcommand{\maxd}[1]{\Delta(#1)}
\newcommand{\mind}[1]{\delta(#1)}
\newcommand{\R}{\mathbb R}

\title{Graphs of 20 edges are $2$--apex, hence unknotted}
\author{Thomas W.\ Mattman}
\address{Department of Mathematics and Statistics,
California State University, Chico,
Chico, CA 95929-0525}
\email{TMattman@CSUChico.edu}

\begin{document}

\begin{abstract}
A graph is $2$--apex if it is planar after the deletion of at most two vertices. Such graphs are
not intrinsically knotted, IK. We investigate the converse, does not IK imply $2$--apex? We determine the simplest possible counterexample, a graph on nine vertices and 21 edges that is neither IK nor $2$--apex. In the process, we show that every graph of 20 or fewer edges is $2$--apex. This provides a new proof that an IK graph must have at least 21 edges. We also classify IK graphs on nine vertices and 21 edges and find no new examples of minor minimal IK graphs in this set.
\end{abstract}

\maketitle

\section{Introduction}
We say that a graph is intrinsically knotted or IK if every tame embedding of the graph in $\R^3$ contains a non-trivially knotted cycle.
Blain, Bowlin, et al.~\cite{BBFFHL} and 
Ozawa and Tsutsumi~\cite{OT} independently discovered an important criterion for intrinsic knotting. Let $H*K_2$ denote the join of the 
graph $H$ and the complete graph on two vertices, $K_2$.

\begin{prop}[\cite{BBFFHL},\cite{OT}]
A graph of the form $H*K_2$ is IK if and only if $H$ is non--planar.
\end{prop}

A graph is called $l$--apex if it becomes planar after the deletion of at most $l$ vertices (and their edges).
The proposition shows that $2$--apex graphs are not IK.

It's known that many non IK graphs are $2$--apex. 
As part of their proof that intrinsic knotting requires 21 edges, 
Johnson, Kidwell, and Michael~\cite{JKM} showed that every triangle-free 
graph on 20 or fewer edges is $2$--apex and, therefore, not knotted. In the current paper, we show 

\begin{theorem} \label{thmain}%
All graphs on 20 or fewer edges are $2$--apex.
\end{theorem}

This amounts to a new proof that

\begin{cor}
An IK graph has at least 21 edges.
\end{cor}

Moreover, we also show

\begin{prop} \label{prop821}%
Every non IK graph on eight or fewer vertices is $2$--apex.
\end{prop}

This suggests the following:

\begin{que}
Is every non IK graph $2$--apex?
\end{que}

We answer the question in the negative by giving an example of a graph, $E_9$, having nine vertices and 21 edges that is neither IK nor $2$--apex. 
(We thank Ramin Naimi~\cite{N} for providing an unknotted embedding of $E_9$, which appears as Figure~\ref{figE9} in Section 3.)
Further, we show that no graph on fewer than 21 edges, no graph on fewer than nine vertices, and no other graph
on 21 edges and nine vertices has this property. In this sense, $E_9$ is the simplest possible counterexample to our Question.

The notation $E_9$ is meant to suggest that this graph is a ``cousin'' to the set of 14 graphs derived from $K_7$ by triangle--Y moves (see \cite{KS}). Indeed, $E_9$ arises from a Y--triangle move on the graph $F_{10}$ 
in the $K_7$ family. Although intrinsic knotting is preserved under triangle--Y moves~\cite{MRS}, it is not, in general, preserved under Y--triangle moves. For example, although $F_{10}$ is derived from $K_7$ by triangle--Y moves and, therefore, intrinsically knotted, the graph $E_9$, obtained by a Y--triangle move on $F_{10}$, has an unknotted embedding.

Our analysis includes a classification of IK and $2$--apex graphs on nine vertices and at most 21 edges. Such a graph is $2$--apex unless it is
$E_9$, or, up to addition of degree zero vertices, one of four graphs derived from $K_7$ by triangle--Y moves~\cite{KS}.
(Here $|G|$ denotes the number of vertices 
in the graph $G$ and $\|G \|$ is the number of edges.)

\begin{prop} \label{prop921$2$--apex}%
Let $G$ be a graph with $|G| = 9$ and $\|G \| \leq 21$. If $G$ is not $2$--apex, then $G$ is either $E_9$ or else one of the following IK graphs: $K_7 \sqcup K_1 \sqcup K_1$, $H_8 \sqcup K_1$, $F_9$, or $H_9$.
\end{prop}

The knotted graphs are exactly those four descendants of $K_7$:

\begin{prop} \label{prop921IK}%
Let $G$ be a graph with $|G| = 9$ and $\|G \| \leq 21$.
Then $G$ is IK iff it is $K_7 \sqcup K_1 \sqcup K_1$, $H_8 \sqcup K_1$, $F_9$, or $H_9$.
\end{prop}

In particular, we find that there are no new minor minimal IK graphs in
the set of graphs of nine vertices and 21 edges.

We remark that a result of Sachs~\cite{S} suggests a similar analysis of $1$--apex  graphs. A graph is intrinsically linked (IL) if every tame embedding includes a pair of non-trivally linked cycles.

\begin{prop}[Sachs]
A graph of the form $H*K_1$ is intrinsically linked if and only if $H$ is non--planar.
\end{prop}

It follows that $1$--apex graphs are not IL and one can ask about the converse. A computer search suggests that the simplest counterexample (a graph that is neither IL nor $1$--apex) in terms of vertex count is a graph on eight vertices and 21 edges whose complement is the disjoint union of $K_2$ and a six cycle. B\"ohme also gave this example as graph $J_1$ in \cite{B}.
In terms of the number of edges, the disjoint union of two $K_{3,3}$'s is a counterexample of eighteen edges. It's straightforward to verify, using methods similar to those presented in this paper, that a counterexample must have at least eight vertices and at least 15 edges. Beyond these observations, we leave open the 

\begin{que}
What is the simplest example of a graph that is neither IL nor $1$--apex?
\end{que}

The paper is organized into two sections following this introduction. 
In the first we prove Theorem~\ref{thmain}.
In the second we prove Propositions~\ref{prop821}, \ref{prop921$2$--apex},
and \ref{prop921IK}.

\section{Graphs on twenty edges}
In this section we will prove Theorem~\ref{thmain},
a graph of twenty or fewer edges is $2$--apex. We will use
induction and break the argument down as a series of six propositions
that, in turn, treat graphs with eight or fewer vertices, nine vertices, ten vertices,  eleven vertices, twelve vertices, and thirteen or more vertices. Following a first subsection where we introduce some useful 
definitions and lemmas, we devote one subsection to each of 
the six propositions.

\subsection{Definitions and Lemmas}
In this subsection we introduce several definitions and three lemmas.
The first lemma and the definitions that precede it are based on
the observation that, in terms of topological properties such 
as planarity, $2$--apex, or IK, vertices of degree less than three can be
ignored.

Let $N(c)$ denote the neighbourhood of the vertex $c$.

\begin{defn}
Let $c$ be a degree two vertex of graph $G$.
Let $N(c) = \{d,e\}$. {\bf Smoothing} $c$ means replacing the vertex $c$ and edges $cd$ and $ce$ with the edge $de$ to obtain a new (multi)graph $G'$. If $de$ was already an edge of $G$, we can remove one of the copies of $de$ to form the simple graph $G''$. We will say $G''$ is obtained from $G$ by {\bf smoothing and simplifying} at $c$.
\end{defn}

We will use $\mind{G}$ to denote the minimal degree of $G$, i.e., the least degree among the vertices of $G$.

\begin{defn}
Let $G$ be a graph. The multigraph $H$ is the {\bf topological simplification}
of $G$ if $\mind{H} \geq 3$ and $H$ is obtained from $G$ by a sequence of the following three moves:
delete a degree zero vertex; delete a degree one vertex and its edge; and smooth a degree two vertex.
\end{defn}

\begin{defn}
Graphs $G_1$ and $G_2$ are {\bf topologically equivalent} if their
topological simplifications are isomorphic.
\end{defn}

The following lemma demonstrates that in our induction it will be enough to consider graphs of minimal degree at least three, $\mind{G} \geq 3$.
For $a$ a vertex of graph $G$, let $G - a$ denote the induced subgraph on the vertices other than $a$: $V(G) \setminus \{a\}$. Similarly, $\Gab$ and $G - a,b,c$ will denote induced subgraphs on $V(G) \setminus \{a,b\}$ and 
$V(G) \setminus \{a,b,c\}$.

\begin{lemma}
\label{lemd12}%
Suppose that every graph with $n>2$ vertices and at most $e$ edges is $2$--apex.
Then the same is true for every graph with $n+1$ vertices, at most $e+1$ (respectively, $e$) edges, and a vertex of
degree one or two (respectively, zero).
\end{lemma}

\begin{proof}
Let $G$ have $n+1$ vertices and $e'$ edges where $n > 2$ and $e' \leq e+1$. 

If $G$ has a degree zero vertex, $c$, we assume further that $e' \leq e$. In this case, deleting $c$ results in a $2$--apex graph 
$G-c$,  i.e., there are vertices $a$ and $b$ such that $G-a,b,c$ is planar. This implies $\Gab$ is also planar so that $G$ is $2$--apex.

If $G$ has a vertex $c$ of degree one, we may delete it (and its edge) to obtain a graph, $G - c$ on $n$ vertices with $e'-1$ edges. 
Again, by hypothesis,  $G-c$ is $2$--apex, so that $G-a,b,c$ is planar for an appropriate choice of $a$ and $b$. 
This means  $\Gab$ is also planar so that $G$ is $2$--apex.

If $G$ has a vertex $c$ of degree two, smooth and simplify that vertex
to obtain the graph $G'$ on $n$ vertices and 
$e'-1$ or $e'-2$ edges. By assumption, there are vertices $a,b$ in $V(G')$ such that $G'-a,b$ is planar. Since $V(G') = V(G) \setminus \{c\}$, $a$ and $b$ are also vertices in $G$. Notice that $\Gab$ is again planar so that $G$ is $2$--apex.
\end{proof}

In showing that all graphs of 20 or fewer edges are $2$--apex, we will frequently investigate a graph $G$ of 20 edges and delete
two vertices to obtain $G' = \Gab$ which we may assume to be non--planar. By the previous lemma, we can assume $G$ has no vertices of degree less than three
(i.e., $\mind{G} \geq 3$). It follows that  $\mind{G'} \geq 1$ 
The following lemma characterises the graphs $\Gab$ of this form.

In the proof we will make use of the Euler characteristic 
$\chi(G) = |G| - \| G \|$, where $| G |$ is the number of vertices and $\| G \|$ the number of edges.

\begin{lemma} 
\label{lemnonpl}%
Let $G$ be a non--planar graph on $n$ vertices, where $n \geq 6$, with $\mind{G} \geq 1$. Then $G$ has at least $n+3 - \lfloor (n-6)/2 \rfloor$ edges.
\end{lemma}

\begin{proof}
First observe that if $G$ is connected, $G$ will have at least 
$n+3$ edges. Indeed, by Kuratowski's theorem, $G$ must have $K_5$ or $K_{3,3}$ as a minor.
If there is a $K_{3,3}$ minor, then we can construct $K_{3,3}$ from $G$
by a sequence of edge deletions and contractions. Since both
$G$ and $K_{3,3}$ are connected, we can arrange for the sequence to pass through a sequence of connected graphs. We will delete any multiple edges that result from an edge contraction so that the intermediate graphs
are also simple. To complete the argument notice that an edge deletion
or contraction can only increase the Euler characteristic. As $\chi(K_{3,3}) = -3$, we conclude that $\chi(G) \leq -3$,
whence $\| G \| \geq n+3$. If, instead $G$ has a $K_5$ minor, 
then, since $\chi(K_5) = -5$, 
a similar argument shows that $\|G \| \geq n+5 > n+3$.

If $G$ is not planar, then it must have a connected component $G'$ for 
which $\chi(G') \leq -3$. Additional components will increase
$\chi(G)$ only if they are trees, i.e., 
$\chi(G) \leq -3 + T$ where $T$ denotes the number of tree components 
of $G$. If $G'$ has at least 
six vertices, then,  as a tree component requires at least two vertices (recall that $\mind{G} \geq 1$), we see that $T \leq \lfloor (n-6)/2 \rfloor$. Thus $\|G\| \geq n+ 3 - T \geq n+3 - \lfloor (n-6)/2 \rfloor$,
as required. If $G'$ doesn't have six vertices, then $G' = K_5$
and $\chi(G') = -5$. In this case, a similar argument shows that
$\|G \|  \geq n+5 - \lfloor (n-5)/2 \rfloor > n+3 - \lfloor (n-6)/2 \rfloor$. 
\end{proof}

\begin{table}[h]
\begin{tabular}{c|cccc} 
    & 9 & 10 & 11 & 12\\ \hline
6   & 1 & 1 \\
7   & 0 & 2 & 9 \\
8   & 0 & 1 & 11 \\
9   & 0 & 0 & 3 \\
10 & 0 & 0 & 1 & 15 \\ 
11 & 0 & 0 & 0 & 3 \\
\multicolumn{2}{c}{ } \\
\end{tabular}
\caption{\label{tblnonpl}%
A count of non--planar graphs with $\mind{G} \geq 1$.  Columns are labelled by the number of edges and rows by the number of vertices.}
\end{table}

\begin{rmk}
\label{rmknonpl}
Table~\ref{tblnonpl} gives the number of graphs satisfying the hypotheses of the lemma. 
Moreover, using the reasoning outlined in the proof of the lemma, 
we can characterise such a non--planar graph $G$ according to the number of vertices as follows.

If $G$ has six vertices and nine edges, then $G = K_{3,3}$. 
If $|G| = 6$ and $\| G \| = 10$, then $G=K_{3,3}$ with one additional edge. 

\begin{figure}[ht]
\begin{center}
\includegraphics[scale=0.33]{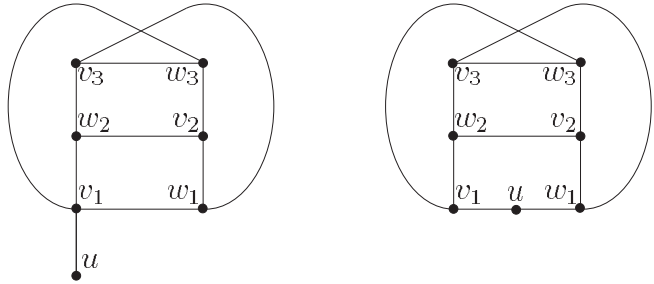}
\caption{Two non--planar graphs with seven vertices and ten edges.}\label{figG710}
\end{center}
\end{figure}

\begin{figure}[ht]
\begin{center}
\includegraphics[scale=0.63]{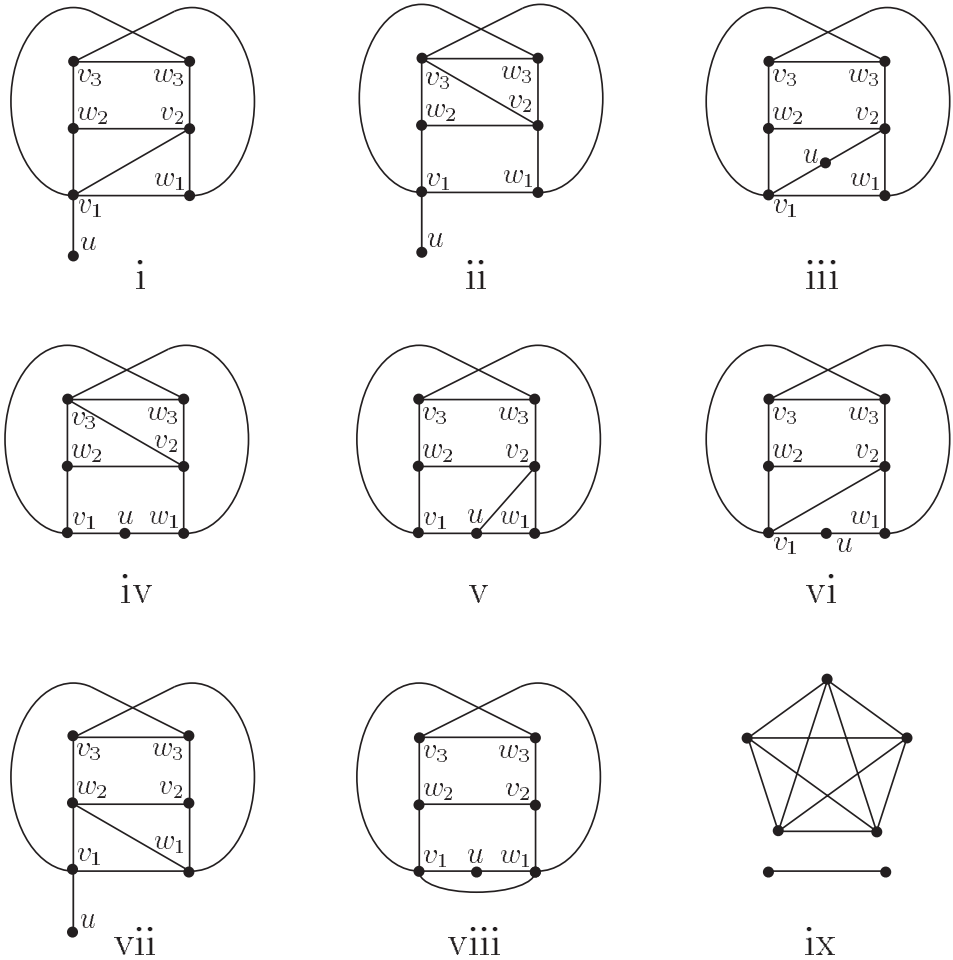}
\caption{The nine non--planar graphs with seven vertices and eleven edges.}\label{figG711}
\end{center}
\end{figure}

If $G$ has seven vertices and ten edges, it is one of the two graphs illustrated in Figure~\ref{figG710} obtained from
$K_{3,3}$ by splitting a vertex.
As for $|G| = 7$ and $\|G \|  = 11$, there are nine such graphs obtained by 
splitting a vertex of a non--planar graph on six vertices or else by adding 
an edge to a graph on ten edges, see Figure~\ref{figG711}.

\begin{figure}[ht]
\begin{center}
\includegraphics[scale=0.6]{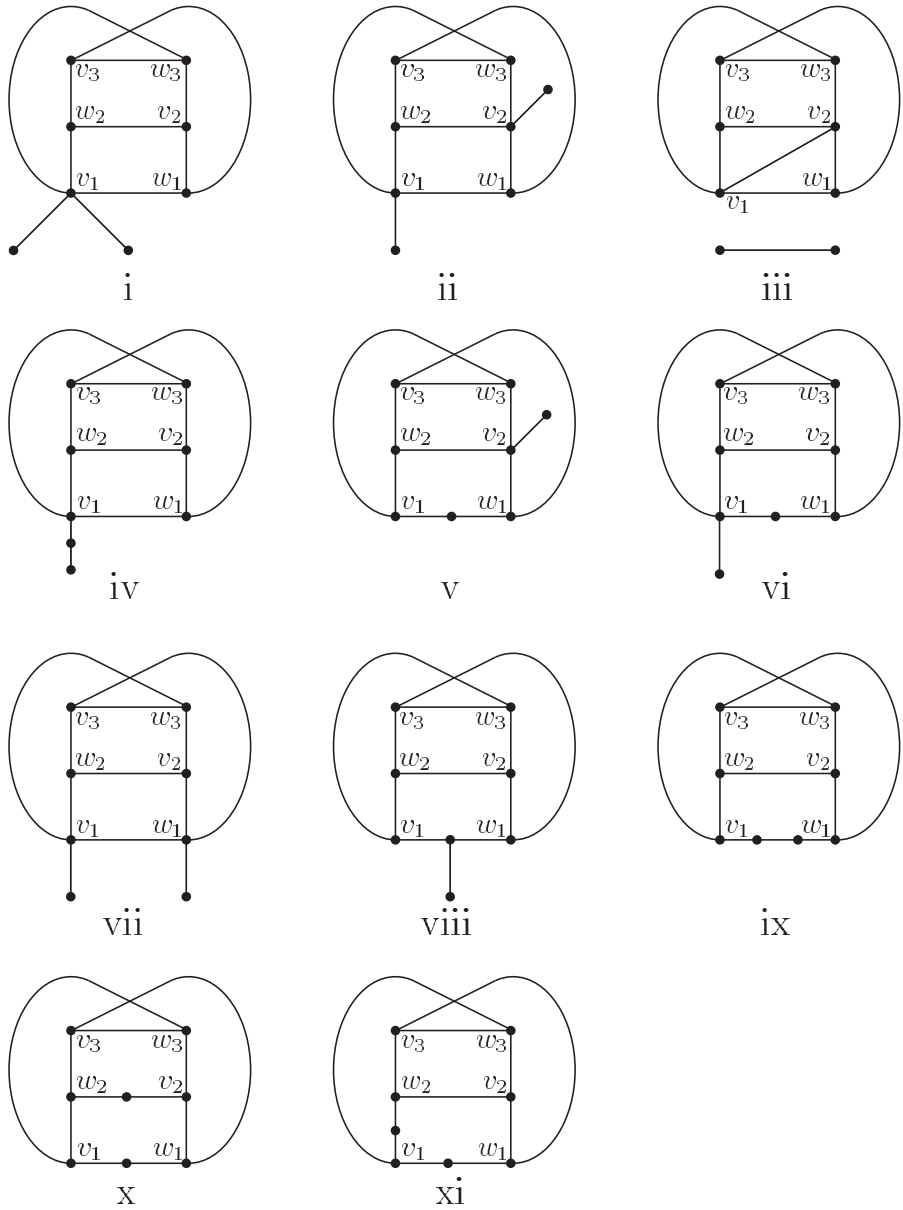}
\caption{Non--planar graphs with eight vertices and eleven edges.}\label{figG811}
\end{center}
\end{figure}

The disjoint union $K_{3,3} \sqcup K_2$ is the only graph $G$ with eight vertices and ten edges. The 11 graphs $G$ with $|G| = 8$ and $\|G\| = 11$ are illustrated in Figure~\ref{figG811}.

Two of the three graphs with $|G| = 9$ and $\| G \| =11$ are formed by the union of $K_2$ with the two graphs having seven vertices and ten edges. The third is the union of $K_{3,3}$ and the tree of two edges.

The unique graph with $|G|=10$ and $\| G \| = 11$ is $K_{3,3} \sqcup K_2 \sqcup K_2$. Of the 15 graphs with $|G| = 10$ and $\| G \| = 12$, 11 are formed by the union of $K_2$ with one of the non--planar graphs on eight vertices and eleven edges, 
two are the union of the tree of 
two edges with a non--planar graph on seven vertices and ten edges, 
and the remaining two are formed by the union of $K_{3,3}$ with the
two trees of three edges.

The graphs with $|G| = 11$ and $\| G \| = 12$ are formed by the union of $K_2$ with a non--planar graph on nine vertices and 11 edges. 
If $G$ has $11$ vertices and $13$ edges, then, it is either $K_5 \sqcup K_2 \sqcup K_2 \sqcup K_2$, or else it
has exactly one tree component, the rest of the graph having a $K_{3,3}$ minor. 
\end{rmk}

Almost all of the graphs mentioned in the remark have $K_{3,3}$ minors. The following definition seeks to take advantage of this.

\begin{defn} Let $G$ be a graph with vertex $v$ and let $v_1,v_2,v_3$ and 
$w_1,w_2,w_3$ denote the vertices in the two parts of $K_{3,3}$. The pair $(G;v)$ is a {\bf generalised ${\mathbf K_{3,3}}$} if the induced subgraph $G-v$ is topologically equivalent to $K_{3,3} - v_1$.
It follows that the vertices of $G-v$ can be partitioned into five disjoint sets $V_2$, $V_3$, $W_1$, $W_2$, and $W_3$, where each of these
five sets induces a tree as a subgraph of $G-v$, such that when each of these trees is contracted down to a single vertex, the tree on $V_i$ becomes the vertex $v_i$ in $K_{3,3} - v_1$ and similarly for the $W_i$. 
When there is a choice of partitions, a partition of a generalised $K_{3,3}$ will be one for which $V_2$ and $V_3$ are minimal. 
\end{defn}

We next observe that when $\Gab$ is a generalised $K_{3,3}$ this will have implications for $N(a)$ and $N(b)$, under the assumption that
$G$ is not $2$--apex.

\begin{lemma} \label{lemgenK33}%
Suppose that $G$ is not $2$--apex and that $(\Gab;c)$ is a generalised $K_{3,3}$.
Then $N(a)$ and $N(b)$ each include at least one vertex from each of
$W_1$, $W_2$, and $W_3$.
\end{lemma}

\begin{proof}
Let $V_2$, $V_3$, $W_1$, $W_2$, and $W_3$ be the partition of the 
vertices of $G-a,b,c$ as in the definition of a generalised $K_{3,3}$. 

Suppose $a$ has no neighbour in $W_1$. Note that by contracting
the subgraphs of $G-b,c$ induced by $V_2$, $V_3$, $W_1$, $W_2$, and $W_3$,
we obtain a (multi)graph $(K_{3,3}-v_1) + a$ formed by adding a vertex $a$ to $K_{3,3} - v_1$. As $a$ is not adjacent to $w_1$, it follows that
$(K_{3,3}-v_1)+a$ has a planar embedding. Now, reversing the contractions
performed earlier, this results in a planar embedding of $G-b,c$, a contradiction. 

Therefore, $a$ has a neighbour in $W_1$. Similarly, $b$ also has a neighbour in $W_1$, and both $a$ and $b$ have neighbours in $W_2$ and $W_3$.
\end{proof}

\begin{rmk} \label{rmkgenK33}%
Lemma~\ref{lemgenK33} also applies (with obvious modifications) when $\Gab$ has a generalised $K_{3,3}$ component with the remaining components being trees. 
\end{rmk}

\subsection{Eight or fewer vertices}

We are now in a position to prove Theorem~\ref{thmain}.
We begin with graphs of eight or fewer vertices. 

\begin{rmk}
In what follows, we will often make use of the following strategy.
To argue that a graph $G$ is $2$--apex, proceed by contradiction. Assume
$G$ is not $2$--apex. This means that every subgraph of the form $\Gab$ is non--planar. Using this assumption we eventually deduce that a particular
$\Gab$ is planar. Although
we won't always say it explicitly, in demonstrating a planar $\Gab$, we have in fact derived a contradiction that shows that $G$ is $2$--apex.
\end{rmk}

\begin{prop} \label{prop820}%
A graph $G$ with $|G| \leq 8$ and $\|G\| \leq 20$ is $2$--apex.
\end{prop}

\begin{proof}
We can assume $|G| \geq 5$ as otherwise $G$ is planar and {\em a fortiori} $2$--apex.  If $|G| \leq 7$, then $G$ is a proper subgraph of $K_7$. So, with an appropriate choice of vertices $a$ and $b$, $\Gab$ is a proper subgraph
of $K_5$ and therefore planar. Thus, $G$ is $2$--apex.

So, we may assume $|G| = 8$ and we will also take $\|G\| = 20$. We will investigate induced subgraphs $\Gab$ formed by deleting two vertices $a$ and $b$. Notice that $a$ and $b$ may be chosen so that $\| \Gab \| \leq 10$. Indeed, the maximum degree of $G$ is at most seven, while the pigeonhole
principle implies the maximum degree is at least five:
$5 \leq \maxd{G} \leq 7$. By Lemma~\ref{lemd12}, the minimum degree is at least three: $\mind{G} \geq 3$. Since $\| G \| = 20$, the sum of the 
vertex degrees is $40$ and it follows that there are vertices $a$ and
$b$ such that $\Gab$ has at most ten edges. 

Assume $G$ is not $2$--apex. Then for each pair of vertices $a$ and $b$, $\Gab$ is not planar. By Lemma~\ref{lemnonpl} such a non--planar $\Gab$ has at least nine edges. Thus, it will suffice to consider the cases where
$G$ has a non--planar subgraph $\Gab$ of nine or ten vertices.
We may assume $d(a) \geq d(b)$.

Suppose first that $\Gab$ is non--planar and has nine edges. By Remark~\ref{rmknonpl}, $\Gab = K_{3,3}$. Let $v_1, v_2, v_3$ be the vertices in
one part of $K_{3,3}$ and $w_1, w_2, w_3$ those in the other.
Since $\|G\| = 20$, $\| \Gab \| = 9$, and $d(a) \geq d(b)$, then $d(a)$ is seven or six. In either case, $\| N(a) \cap N(b) \cap V(\Gab) \| \geq 3$, so
we can assume $v_1$ and $v_2$, say, are in the intersection.
If $d(a) = 7$, it follows that $G-a,v_1$ is planar and $G$ is $2$--apex. If $d(a) = 6$, by Lemma~\ref{lemgenK33}, 
$\{w_1, w_2, w_3 \} \subset N(b)$. But then,
since $\| N(a) \cap V(\Gab) \| \geq 5$, we can assume $aw_1 \in E(G)$ 
(i.e., $aw_1$ is an edge of $G$) and it follows that $G - a, w_1$ is planar whence $G$ is $2$--apex.

Next suppose $\Gab$ is non--planar and has ten edges. That is, by Remark~\ref{rmknonpl}, $\Gab$ is $K_{3,3}$ with
an extra edge. Again, $v_i$ and $w_i$, ($i = 1,2,3$) will denote 
the vertices in the two parts of $K_{3,3}$ and let $v_1v_2$ be the additional edge. 

Suppose first that $d(a) = 5$. This implies $d(b) = 5$, $ab \not\in E(G)$, and there are four or five elements in
$N(a) \cap N(b)$. If five, then $G$ has $K_{3,3}$ as an induced subgraph
after deleting two vertices, a case we considered earlier. 
So, we can assume 
there are four vertices in the intersection, including at least one 
of the vertices $v_1, v_2, v_3$, call it $v$ and at least one $w_i$ vertex, say $w_1$. Then, $G-v,w_1$ is planar and $G$ is $2$--apex.

So, we can assume $d(a) > 5$.
By Lemma~\ref{lemgenK33}, $\{w_1, w_2, w_3 \} \subset N(b)$.
In that case, without loss of generality, $aw_1 \in E(G)$. Then $G - a, w_1$ is planar and $G$ is $2$--apex. This completes
the argument when $\Gab$ has ten edges.

We have shown that when $|G| = 8$ and $\|G \| = 20$, $G$ is $2$--apex. It follows
that graphs having $|G| = 8$ and $\| G \| \leq 20$ are also $2$--apex.
\end{proof}

\subsection{Nine vertices}

In this subsection we prove Theorem~\ref{thmain} in the case of graphs
of nine vertices. We begin with a lemma.

\begin{lemma} \label{lem91}%
Let $G$ be a graph with $|G|= 9$, $\| G \| = 20$, $\mind{G} \geq 3$, 
$\maxd{G} = 5$, and such that all degree five vertices are mutually adjacent. Then $G$ is $2$--apex.
\end{lemma}

\begin{proof}
The degree bounds imply that $G$ has four, five, or six degree five vertices. If $G$ has six degree five vertices, then, as they are 
mutually adjacent, $G$ has a $K_6$ component. This implies the other
component, on three vertices, has at most three edges and the graph has
at most 18 edges in total, which is a contradiction. So, in fact, 
$G$ cannot have six degree five vertices.

If $G$ has five degree five vertices, then the induced subgraph
on the other four vertices has five edges, so it is $K_4 - e$ 
($K_4$ with a single edge deleted). Let $c$ be a degree four vertex that
has degree two in the induced subgraph $K_4 - e$ and let 
$a$ and $b$ be the two degree five neighbours of $c$. Then
$\Gab$ is planar and $G$ is $2$--apex.

\begin{figure}[ht]
\begin{center}
\includegraphics[scale=0.2]{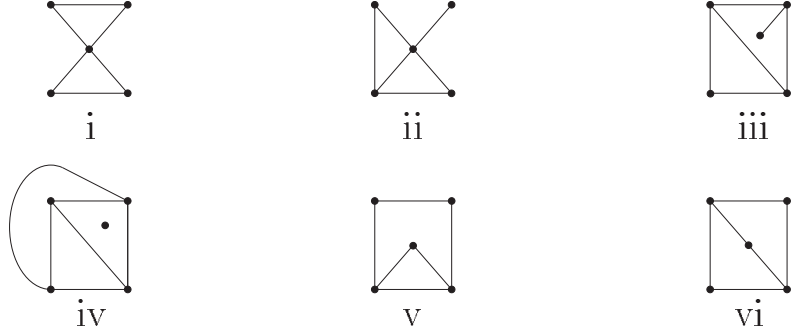}
\caption{The six graphs of six edges on five vertices.}\label{fig9f1}
\end{center}
\end{figure}

If $G$ has four degree five vertices, then the induced subgraph
on the other five vertices has six edges, so it is one of the six
graphs in Figure~\ref{fig9f1}. For graphs i, ii, and iii, the argument is similar to the previous case. That is, 
let $c$ be a degree four vertex of $G$ that
has degree two in the induced subgraph and let $a$ and $b$ be the degree
five neighbours of $c$. Then $\Gab$ is planar and $G$ is $2$--apex. For graph iv,
if $a$ and $b$ are any of the degree five vertices, $\Gab$ is planar and $G$ is $2$--apex. For graphs v and vi, the argument is a little more involved, but, 
again, there are vertices $a$ and $b$ such that $\Gab$ is planar and
$G$ is $2$--apex.
\end{proof}

We are now ready to prove Theorem~\ref{thmain} in the case of nine vertices.

\begin{prop} \label{prop920}%
A graph $G$ with $|G| = 9$ and $\|G\| \leq 20$ is $2$--apex.
\end{prop}

\begin{proof} First, we'll assume $\|G \| = 20$. Then, $5 \leq \maxd{G} \leq 8$ and, by Lemma~\ref{lemd12},
$\mind{G} \geq 3$. If $\maxd{G} > 5$, by appropriate choice of vertices $a$ and $b$, $\Gab$ has at most ten edges.
This is also true when $\maxd{G} = 5$ unless all degree 5 vertices
are mutually adjacent. As Lemma~\ref{lem91} treats that case, we may
assume that there is a $\Gab$ of at most ten edges. Moreover, we'll take
$d(a) \geq d(b)$.

Assuming $G$ is not $2$--apex, then that $\Gab$ is non--planar. By Remark~\ref{rmknonpl}, $\Gab$ is one of the two graphs in Figure~\ref{figG710}. 
Suppose first that it is the graph at left in the figure. As $u$ has degree three or more in $G$, both $a$ and $b$ are adjacent to $u$. 
By Lemma~\ref{lemgenK33}, $\{ w_1, w_2, w_3 \} \subset N(b)$. Without loss of generality, we can assume $aw_1 \in E(G)$. Then $G - a, w_1$ is planar and $G$ is $2$--apex.

Suppose, then, that $\Gab$ is the graph at right in Figure~\ref{figG710}.
By Lemma~\ref{lemgenK33}, $\{ w_2, w_3 \} \subset N(a) \cap N(b)$ and at least one of $w_1$ or $u$ is a neighbour of each
$a$ and $b$. Now, as $G$ is not $2$--apex, $G - w_2, w_3$ is non--planar and it is also a graph on seven vertices and ten edges with either $u$ or $w_1$ of degree 
at least four. In other words, $G - w_1, w_2$ is the graph
on the left of Figure~\ref{figG710}, a case we considered earlier. 

We have shown that if $\|G \| = 20$, then $G$ is $2$--apex. It follows that 
the same is true for graphs with $\|G \| \leq 20$.
\end{proof}

\subsection{Ten vertices}

In this subsection we prove Theorem~\ref{thmain} for graphs of ten vertices. 
We begin with a lemma that treats the case of a graph of degree four.

\begin{lemma} \label{lem104}%
Suppose $G$ is a graph with $|G| = 10$, $\|G\| = 20$, and such that every vertex has degree four. Then $G$ is $2$--apex.
\end{lemma}

\begin{proof}
We can assume that $G$ has at least three vertices $a$, $b$, and $c$ that are pairwise non--adjacent for otherwise $G$ must be $K_5 \sqcup K_5$ and is $2$--apex. Then $\maxd{\Gab} = 4$
as $c$ will retain its full degree in $\Gab$. Also, $\mind{\Gab} = 2$; since $c \not\in N(a) \cup N(b)$, $a$ and $b$ must share at least
one neighbour in the remaining seven vertices. This will become a degree
two vertex in $\Gab$. 

Now, $\Gab$ is a graph on eight vertices and 12 edges with at least one degree two vertex. Smoothing that vertex, we arrive at $G'$, a multigraph on seven vertices and 11 edges that we can take to be non--planar (otherwise $G$ is $2$--apex). In other words, $G'$ is either one of
the graphs in Figure~\ref{figG711} or else one of the two graphs in
Figure~\ref{figG710} with an edge doubled. Moreover, $\maxd{G'} = 4$ and
$\mind{G'} \geq 2$. Examining these candidates for $G'$, we see that
$\Gab$ has degree sequence $\{4,3,3,3,3,3,3,2\}$ 
or $\{4,4,3,3,3,3,2,2\}$.

\begin{figure}[ht]
\begin{center}
\includegraphics[scale=0.6]{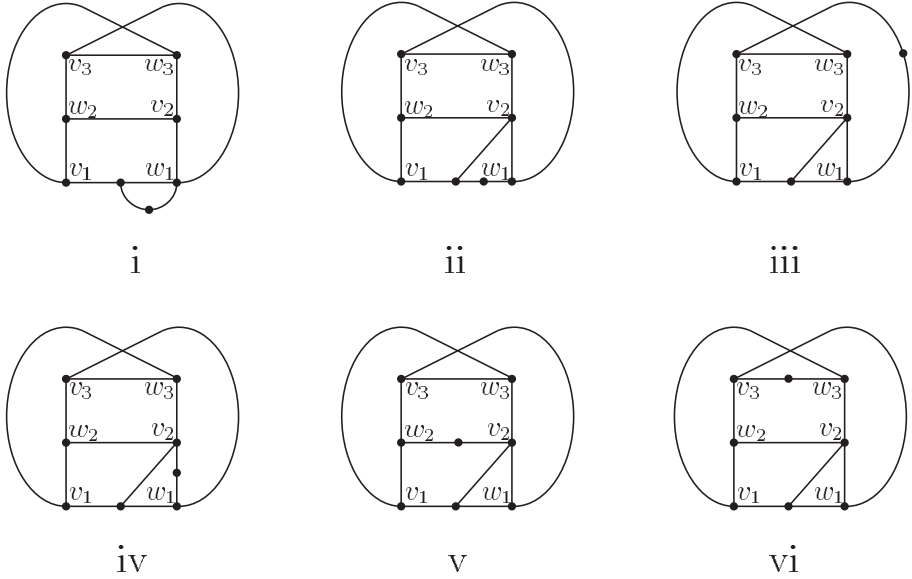}
\caption{The six non--planar graphs with degree sequence $\{4, 3,3,3,3,3,3,2\}$.}\label{figG842}
\end{center}
\end{figure}

The six non--planar graphs with degree sequence $\{4,3,3,3,3,3,2\}$ 
(see Figure~\ref{figG842}) are
obtained by either doubling an edge at $u$ in the graph on the right of 
Figure~\ref{figG710} or else by adding a degree two vertex to graph v of 
Figure~\ref{figG711}. If $\Gab$ is one of the graphs ii, iii, or iv in Figure~\ref{figG842}, then we argue that $G$ is $2$--apex as follows. 
By applying Lemma~\ref{lemgenK33} to $(\Gab ; v_2)$, we find $\{w_2,w_3\} \subset N(a) \cap N(b)$.
But then $d(w_2) = d(w_3) = 5$, contradicting our hypothesis
that all vertices have degree four. A similar argument (using $(\Gab; w_1)$ and $v_2,v_3$ in place of $w_2,w_3$)
 applies when $\Gab$ is graph i. For graphs v and vi, the same approach shows that at least one of $w_2$ and $w_3$ has degree five. The contradiction shows that
$G$ is $2$--apex in case $\Gab$ has degree sequence $\{4,3,3,3,3,3,2\}$.

So, we may assume $\Gab$ has degree sequence $\{4,4,3,3,3,3,2,2\}$. Then
$\Gab$ is either obtained by doubling an edge of the graph at right in 
Figure~\ref{figG710} or else by adding a degree two vertex to graph iii, iv, vi, or viii of Figure~\ref{figG711}. 

Suppose first that $\Gab$ comes from doubling an edge of the right graph of Figure~\ref{figG710} (and adding a degree two vertex to one of the two edges in the double). Up to symmetry, the doubled edge is either $v_1 w_2$ or else
$v_2 w_2$. In either case, ($\Gab$; $w_2$) is a generalised $K_{3,3}$, 
whence $v_3 \in N(a) \cap N(b)$.
But then $d(v_3) = 5$ in contradiction to our hypotheses. So $G$ is $2$--apex in this case. 

Finally, to complete the proof, suppose
$\Gab$ is graph iii, iv, vi, or viii of Figure~\ref{figG710}. 
The strategy here is similar to the previous case. We identify a degree four vertex, $c$, of $\Gab$, ($c$ is $v_2$, except for graph viii in which case $c$ is $v_1$) and observe that $(\Gab;c)$ is a generalised
$K_{3,3}$. We then find a vertex $x$ (either $w_2$ or $w_3$ depending on the placement of the degree two vertex) that must lie in $N(a) \cap N(b)$. 
Consequently $d(x) = 5$, a contradiction.
The contradiction shows that $G$ is $2$--apex.
\end{proof}

We can now prove Theorem~\ref{thmain} for graphs of ten vertices.

\begin{prop} A graph $G$ with $|G| = 10$ and $\|G \| \leq 20$ is $2$--apex.
\end{prop}

\begin{proof}
Suppose $|G|=10$ and $\|G \| = 20$. Then $9 \geq \maxd{G} \geq 4$. 
By Lemma~\ref{lemd12}, we can take $\mind{G} \geq 3$ and by Lemma~\ref{lemnonpl}, if $\Gab$ is non--planar, it has at least ten edges.
So, we may assume $\maxd{G} \leq 7$ as, otherwise, there are vertices
$a$ and $b$ so that $\| \Gab \| < 10$ whence $G$ is $2$--apex.

If $\maxd{G} = 7$, then $G$ is $2$--apex unless every subgraph $\Gab$ has at least ten edges. So, we can assume $G$ has degree sequence $\{ 7, 4,4,4,4,4,4,3,3,3 \} $ with each of the degree four vertices adjacent to the
vertex, $a$, of degree seven. For almost all choices of $b$, $\| \Gab \| = 10$ so that, by Remark~\ref{rmknonpl}, $\Gab = K_{3,3} \sqcup K_2$. Then 
$\Gab$ has two degree one vertices which must arise from degree three vertices of $G$ from which two edges have been deleted. This implies $a$
is adjacent to at least two degree three vertices in $G$. This is a contradiction as $N(a)$ includes only one degree three vertex, the remaining six vertices being those of degree four. The contradiction shows
that $G$ is $2$--apex in case $\maxd{G} = 7$.

If $\maxd{G} = 4$, then, in fact every vertex of $G$ has degree four. 
This case is treated in Lemma~\ref{lem104}. Thus, the remainder of 
this proof treats the case where $\maxd{G} = 6$ or $5$. Then there are vertices $a$ and $b$ such that $ \| \Gab \| \leq 11$. By Remark~\ref{rmknonpl} we may assume $\Gab$ is either $K_{3,3} \sqcup K_2$ or else one of the graphs in Figure~\ref{figG811}. Further, we will assume 
$\maxd{G} = d(a) \geq d(b)$.

Suppose $\Gab$ is $K_{3,3} \sqcup K_2$ and let $v_1, v_2, v_3$ and $w_1,w_2,w_3$ be the vertices in the two parts of $K_{3,3}$ while $u_1,u_2$
will denote the vertices of $K_2$.  By Remark~\ref{rmkgenK33},
$(K_{3,3};v_1)$ shows
$\{w_1, w_2, w_3 \} \subset N(b)$. Similarly, $(K_{3,3}; w_1)$ implies $\{v_1, v_2, v_3 \} \subset N(b)$. Finally, as $u_1$ and $u_2$ have degree one in 
$\Gab$, both must be adjacent to $b$ in $G$. This implies $d(b) \geq 8$ 
which contradicts our assumption that $\maxd{G} \leq 6$. The contradiction
shows that $G$ is $2$--apex in case it has a subgraph of the form $K_{3,3} \sqcup K_2$.

We may now assume that $\| \Gab \| = 11$ and that for any other pair
$a'$, $b'$, $\| G - a',b' \| \geq 11$. This allows us to dismiss the
case where $\maxd{G} = d(a) = 6$. Indeed, the condition $\| G - a',b' \| \geq 11 $ then implies that the other vertices of $G$ have degree at most four
and each degree four vertex is adjacent to $a$. But then $G$ would have degree sequence $\{ 6, 4,4,4,4,4,4,4,3,3 \} $ and there are too many 
degree four vertices for them all to be adjacent to $a$. The contradiction
shows that $G$ is $2$--apex in this case.

Suppose then that $\maxd{G} = 5$, $\mind{G} \geq 3$, and that for
every choice of $a'$ and $b'$, $\| G - a', b' \| \geq 11$. Further, let $a$ and $b$ be vertices such that $\| \Gab \| = 11$.
Then $\Gab$ is one of the graphs in 
Figure~\ref{figG811} and we can assume that $d(a) = 5$. The following argument applies to all but the last two graphs in the figure. 

By Lemma~\ref{lemgenK33} (or Remark~\ref{rmkgenK33}), $\{w_2, w_3 \} \subset N(a) \cap N(b)$.
However, either this is already a contradiction because $w_2$ or $w_3$ now has degree greater than
$\maxd{G} = 5$, or else, $d(w_2) = d(w_3) = 5$. In the latter case, as $w_2 w_3 \not\in E(G)$ then $\| G - w_2, w_3 \| = 10$, contradicting our assumption that $\| G - a',b'\| \geq 11$. The contradiction shows that $G$
is $2$--apex.

Similar considerations show that if $\Gab$ is graph x or xi of Figure~\ref{figG811}, then, again, $G$ must be $2$--apex. This completes the argument
in the case that $\|G \| = 20$. 

We have shown that if $\|G \| = 20$, then $G$ is $2$--apex. It follows that 
the same is true for graphs with $\|G \| \leq 20$.
\end{proof}

\subsection{Eleven vertices}

In this subsection, we prove Theorem~\ref{thmain} for graphs of 11 vertices.
We begin with a lemma that handles the case where $\maxd{G} = 4$.

\begin{lemma} \label{lem114}%
Let $G$ have $|G| = 11$, $\|G\| = 20$, and $\maxd{G} = 4$.
Then $G$ is $2$--apex.
\end{lemma}

\begin{proof}
By Lemma~\ref{lemd12}, we can take $\mind{G} \geq 3$ so that $G$ has degree
sequence $\{4,4,4,4,4,4,4,3,3,3,3\}$. Let $a$ and $b$ be two non--adjacent vertices of degree four. Then $\Gab$ has nine vertices and 12 edges. 
Since $\| \Gab \| = 12$ and $\mind{\Gab} \geq 1$, we see that 
$\Gab$ has at least two vertices of degree less than two. Deleting or smoothing those two, we arrive at a multigraph $G'$ with seven vertices and
ten edges. We can assume $G'$ is non--planar as otherwise $\Gab$ is planar and $G$ is $2$--apex. Thus $G'$ is either one of the two graphs in Figure~\ref{figG710}, $K_{3,3} \sqcup C_1$ where $C_1$ is a loop on a single vertex, 
$K_5 \sqcup K_1 \sqcup K_1$, or else the union of $K_1$ and $K_{3,3}$ with an extra edge. We will consider these five possibilities in turn.

If $G'$ is $K_5 \sqcup K_1 \sqcup K_1$, then $\Gab = K_5 \sqcup K_2 \sqcup K_2$. In order to bring the four degree one vertices of $\Gab$ up to degree three in $G$, each must be adjacent to both $a$ and $b$. Then the induced subgraph on $a$, $b$, and the vertices of the two $K_2$'s is planar so that $G$ is not only $2$--apex, it's actually $1$--apex.

Suppose next that $G'$ is the union of $K_1$ and $K_{3,3}$ with an extra edge.
Let $v_1,v_2,v_3$ and $w_1,w_2,w_3$ be the vertices in the two parts of
$K_{3,3}$. Without loss of generality, the extra edge of $K_{3,3}$ is
either $v_1w_1$ (doubling an existing edge) or else $v_1 v_2$. 
By Remark~\ref{rmkgenK33}, $a$ and $b$ both have neighbours in the
three sets $W_1$, $W_2$, and $W_3$. Moreover, at least one of these three
sets consists of a single vertex $w$. But then $d(w) = 5$, a contradiction. 
The contradiction shows that $G$ is $2$--apex in this case.
If $G' = K_{3,3} \sqcup C_1$ or $G'$ is the graph at the left of Figure~\ref{figG710}, the same argument applies and we conclude $G$ is $2$--apex.

Now, if $G'$ is the graph at the right of Figure~\ref{figG710}, then $u$ is a degree two vertex near $w_1$ (so that $W_1$ includes at least those
two vertices) and the additional two degree one and two vertices might lie near $w_2$ and $w_3$ so that in the generalised $K_{3,3}$, $(G'; v_1)$,  none of the $W_i$'s is a single vertex.
For example, $\Gab$ may be graph i of Figure~\ref{figG933222} below. Actually, we can conclude that $\Gab$ must be graph i. For otherwise, 
examining $(\Gab;v)$ in turn for all choices of vertex $v$, 
we will discover at least one $v_i$ or $w_i$ vertex, call it $w$, that must 
lie in $N(a) \cap N(b)$ which leads to the contradiction that $d(w) = 5$. 

Thus, we are left to consider the case where $\Gab$ is graph i of Figure~\ref{figG933222} below. Each of the three vertices $u_1$, $u_2$, and $u_3$ is adjacent to at least  one of $a$ and $b$ as the $u_i$'s must have degree at least three in $G$. Without loss of generality, we can assume $u_1$ and $u_2$ are neighbours of $a$. Also, $N(a)$ must include at least one vertex from the six $v_i$ and $w_i$ vertices. Up to symmetry, this gives two cases: $\{u_1, u_2, v_1\} \subset N(a)$  and $\{u_1, u_2, v_3\} \subset N(a)$. 

Suppose first that $\{u_1, u_2, v_1\} \subset N(a)$. Then in the 
generalised $K_{3,3}$, $(\Gab;v_1)$,
$W_3 = \{w_3, u_3 \}$ and $W_3 \cap N(a) \neq \emptyset$. But, if $aw_3 \in E(G)$, then $G - b, w_3$ is planar. So we can assume that $N(a) = \{u_1, u_2, u_3, v_1 \}$. Note that $v_1 \not\in N(b)$ for otherwise $d(v_1) = 5$, contradicting our assumption about the maximum degree of $G$. Also, we've assumed that $ab \not\in E(G)$. Then $G - u_2,u_3$ is planar unless $bu_1 \in E(G)$. Similarly, $G - u_1,u_3$ and $G - u_1, u_3$ show that we can assume $u_2, u_3 \in N(b)$. Now, up to symmetry, we can assume that the 
fourth vertex of $N(b)$ is either $v_2$, $w_1$, or $w_2$, so we consider those three cases.
If $N(b) = \{u_1,u_2,u_3,v_2\}$ then $G - u_3, v_3$ is planar and $G$ is $2$--apex.
If $N(b) = \{u_1,u_2,u_3,w_1\}$ then $G - u_3, v_3$ is planar and $G$ is $2$--apex.
If $N(b) = \{u_1,u_2,u_3,w_2\}$ then $G - u_1, v_1$ is planar and $G$ is $2$--apex. 

The argument in the case that $\{ u_1, u_2, v_3 \} \subset N(a)$ is similar.
\end{proof}

Having treated the case where $\maxd{G} = 4$,
we are ready to prove Theorem~\ref{thmain} for 
graphs of 11 vertices.

\begin{prop} A graph $G$ with $|G| = 11$ and $\|G \| \leq 20$ is $2$--apex.
\end{prop}

\begin{proof}
Suppose $|G|=11$ and $\|G \| = 20$. Then $10 \geq \maxd{G} \geq 4$. 
By Lemma~\ref{lemd12}, we can take $\mind{G} \geq 3$ and by Lemma~\ref{lemnonpl}, if $\Gab$ is non--planar, it has at least 11 edges.
So, we may assume $\maxd{G} \leq 6$ as, otherwise, there are vertices
$a$ and $b$ so that $\| \Gab \| < 11$ whence $G$ is $2$--apex.
Lemma~\ref{lem114} deals with graphs having $\maxd{G} = 4$ and 
we treat the case of $\maxd{G} = 6$ in the following paragraph.

Suppose $\maxd{G} =6$ and let $a$ be a vertex of maximum degree. 
If $G$ is not $2$--apex, then, to meet the requirement that $\| \Gab \| \geq 11$
for every choice of $b$, the remaining vertices have degree three or
four with all degree four vertices adjacent to $a$. It follows that $G$ 
has degree sequence $\{6, 4, 4, 4, 4, 3, 3, 3, 3, 3, 3\}$. Then $a$ is 
adjacent to exactly two of the degree three vertices, call them $c$ and $d$. Thus $N(c) \cup N(d)$ consists of at most four other vertices beside $a$. Let $b$ be a vertex not in $N(c) \cup N(d)$. Then $\Gab$ has 
11 edges and no degree one vertex. By Remark~\ref{rmknonpl}, $\Gab$ is planar and $G$ is $2$--apex. 

So, for the remainder of the proof, we assume
$\maxd{G} = 5$. If $G$ is not $2$--apex, then, the condition $\| \Gab \| \geq 11$ implies all degree five vertices are mutually adjacent. Moreover, 
either there are vertices $a$ and $b$ with $d(a) = d(b) = 5$ and $\| \Gab \| = 11$, or else $G$ has degree sequence 
$\{ 5, 4, 4, 4, 4, 4, 3, 3, 3, 3, 3 \}$.

Suppose, first, that $\| \Gab \| = 11$ with $d(a) = d(b) = 5$. Assuming
$G$ is not $2$--apex,
by Remark~\ref{rmknonpl}, $\Gab$ is one of three graphs. 
If $\Gab$ is the union of the graph at the left of Figure~\ref{figG710} and $K_2$, then $a$ must be adjacent to each of the three degree one vertices of $\Gab$ as otherwise they will have degree at most two in $G$.
By Remark~\ref{rmkgenK33}, $\{w_1, w_2, w_3\} \subset N(a)$ which implies
$d(a) \geq 6$, a contradiction. So $G$ is $2$--apex in this case.
If $\Gab$ is either the union of the graph at the right of the figure and $K_2$ or else the union of $K_{3,3}$ and a tree on three vertices, again, $a$ must be adjacent to the two degree one vertices in the tree.
But, by Remark~\ref{rmkgenK33}, $\{ v_2, v_3, w_2, w_3 \} \subset N(a)$.
This again gives the contradiction $d(a) \geq 6$, which shows that $G$ is $2$--apex in this case as well. 

Thus, we can assume that $G$ has degree sequence $\{5, 4, 4, 4, 4, 4, 3, 3, 3, 3, 3\}$.
Further, we can assume all the degree four vertices are adjacent to $a$, the vertex of degree five. 
For otherwise, let $b$ be a degree four vertex not adjacent to $a$. Then $\| \Gab \| = 11$ so it is one of the three graphs mentioned in Remark~\ref{rmknonpl}, each of which has two
degree one vertices. As $b$ is adjacent to all the degree one vertex, it has at most two neighbours in $\{v_1, v_2, v_3, w_1,w_2,w_3\}$. That would imply
$\Gab$ is planar, a contradiction.

So, let $a$ be adjacent to all the degree four vertices.
Then $G - a$ has all vertices of degree three and, for any vertex $b$, $\Gab$ has degree sequence $\{3,3,3,3,3,3,2,2,2 \}$. Smoothing one of the degree two vertices, we have the multigraph $G'$ with $|G'| = 8$ and $\| G' \| = 11$. If $G$ is not $2$--apex, then $G'$ is non--planar and,
by Remark~\ref{rmknonpl}, is either $K_{3,3} \sqcup K_2$ with one 
edge doubled or else it is one of the graphs of Figure~\ref{figG811} 
with an additional degree two vertex. Then $\Gab$ is either $K_{3,3} \sqcup C_3$, where $C_3$ is the cycle of three vertices, or else $\Gab$ is $K_{3,3}$ with the addition of three degree two vertices. However, if
$\Gab$ is $K_{3,3} \sqcup C_3$ we deduce that
$G - a$ is $K_{3,3} \sqcup K_4$. Let $v_1$ be one of the vertices of $K_{3,3}$, then $G - a,v_1$ is planar and $G$ is $2$--apex. 

\begin{figure}[ht]
\begin{center}
\includegraphics[scale=0.65]{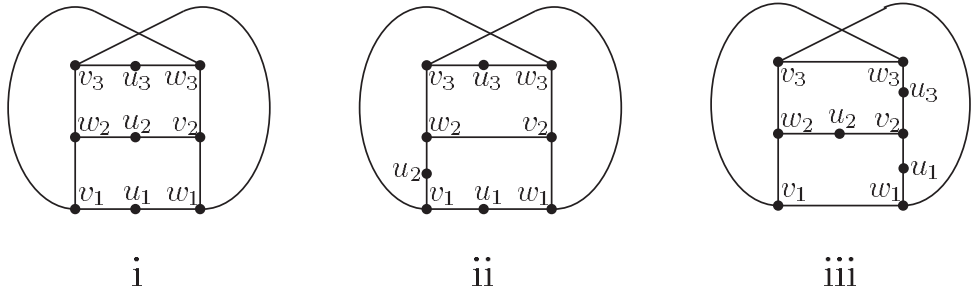}
\caption{Three non--planar graphs with degree sequence $\{3,3,3,3,3,3,2,2,2\}$.}\label{figG933222}
\end{center}
\end{figure}

So, we can assume $\Gab$ is $K_{3,3}$ with the addition of three degree 
two vertices. Let $v_1, v_2, v_3$ and $w_1, w_2, w_3$ denote the vertices in the two parts of $K_{3,3}$ as well as the corresponding vertices
in $\Gab$. Suppose the degree two vertices are all on the edges, $v_1w_1$, $v_1w_2$, and $v_2w_1$ of $K_{3,3}$. Then $G - a, v_3$ is planar so that $G$ is $2$--apex. Thus, we can assume $\Gab$ is one of the three graphs in Figure~\ref{figG933222}. Now, if $\Gab$ is graph ii or iii, then $G - a, w_3$ is planar and $G$ is $2$--apex. So, the remainder of the proof treats the case of graph i.

Assume then that $\Gab$ is graph i of Figure~\ref{figG933222} and that
$G$ is not $2$--apex. Further, let  $ab \not\in E(G)$. Since $G - u_1, u_2$ is non--planar, then $au_3 \in E(G)$ and by removing the pairs $u_1, u_3$ and $u_2, u_3$ in turn, we see that we can assume that $a$ is adjacent to $u_1$, $u_2$, and $u_3$. Then $a$ is adjacent to exactly two vertices of $K_{3,3}$, without loss of generality, either $v_1, v_2$; $v_1, w_1$; or $v_1, w_2$. Let us examine these three subcases in turn. If $N(a) = \{u_1, u_2, u_3, v_1, v_2 \}$, then $G - u_2, v_2$ is planar and $G$ is $2$--apex. 
If $N(a) = \{u_1, u_2, u_3, v_1, w_1 \}$, then $G - u_2, v_1$ is planar and $G$ is $2$--apex. 
If $N(a) = \{u_1, u_2, u_3, v_1, w_2 \}$, then $G - u_3, w_3$ is planar and $G$ is $2$--apex.
This completes the argument in case $\Gab$ is graph i of 
Figure~\ref{figG933222} and with it the case of a graph $G$ of twenty edges.

We have shown that if $\|G \| = 20$, then $G$ is $2$--apex. It follows that 
the same is true for graphs with $\|G \| \leq 20$.
\end{proof}  

\subsection{Twelve vertices}

In this subsection we prove Theorem~\ref{thmain} in the case of 
a graph of 12 vertices.

\begin{prop} A graph $G$ with $|G| = 12$ and $\|G \| \leq 20$ is $2$--apex.
\end{prop}

\begin{proof}
Suppose $|G|=12$ and $\|G \| = 20$. Then $11 \geq \maxd{G} \geq 4$. 
By Lemma~\ref{lemd12}, we can take $\mind{G} \geq 3$ and by Lemma~\ref{lemnonpl}, if $\Gab$ is non--planar, it has at least 11 edges.
So, we may assume $\maxd{G} \leq 6$ as, otherwise, there are vertices
$a$ and $b$ so that $\| \Gab \| < 11$ whence $G$ is $2$--apex.

In fact, we can assume $\maxd{G} \leq 5$. Indeed, suppose instead
$\maxd{G}= 6$ with $a$ a vertex of maximum degree. As there are only
twenty edges in all, there must be a degree three vertex $b$ not adjacent to $a$. Then $\| \Gab \| = 11$. If $G$ is not $2$--apex, then, by Remark~\ref{rmknonpl}, $\Gab = K_{3,3} \sqcup K_2 \sqcup K_2$. However, as $d(b) = 3$,
$\Gab$ can have at most three degree one vertices. The contradiction 
shows that $G$ is $2$--apex when $\maxd{G} = 6$.

Let $\maxd{G} = 5$ and suppose that $G$ has two degree five vertices $a$ and $b$. Assuming $G$ is
not $2$--apex, then $\Gab$ is non--planar. By Remark~\ref{rmknonpl}, $a$ and $b$ are
adjacent and $\Gab = K_{3,3} \sqcup K_2 \sqcup K_2$. It follows that
each of $a$ and $b$ is adjacent to each of the four degree one vertices in
$\Gab$ as these vertices come to have degree three in $G$. In particular,
the induced subgraph on $a$, $b$, and the vertices of the two $K_2$'s is planar. If $v_1$ is a vertex in the $K_{3,3}$ component of $\Gab$, then
$G - v_1$ is planar so that $G$ is $1$--apex and, therefore, also $2$--apex.

So, we can assume $G$ has exactly one degree five vertex $a$. 
It follows that
$G$ has exactly two degree four vertices with the remaining vertices of degree three. We can assume that both degree four vertices are adjacent to
$a$ as otherwise a similar argument to that of the last paragraph shows
that $G$ is $1$--apex. Let $b$ be one of the degree four vertices. Then 
$\| \Gab \|  = 12$. Assuming $G$ is not $2$--apex, then $\Gab$ is non--planar and
therefore one of the 15 graphs described in Remark~\ref{rmknonpl}.
However, as $a$ is adjacent to the two degree four vertices,
we see that $\maxd{\Gab} = 3$ which leaves seven candidate graphs:
the union of $K_2$ with graph viii, ix, x, or xi of Figure~\ref{figG811};
the union of the tree on two edges with the graph to the right
in Figure~\ref{figG710}; or $K_{3,3}$ union a tree on three edges.
(There are two such trees.) We will consider each possibility in turn.

If $\Gab$ is $K_2 \sqcup H$ where $H$ is graph ix, x, or xi of Figure~\ref{figG811}, then we deduce that $a$ is adjacent to one of the degree three
vertices of $H$, call it $v$, as that is the only way to produce a second degree four vertex in $G$ (besides $b$). We claim that $G - a, v$ is planar. Indeed, $\| G - a,v \| = 12$. But $G - a,v$ is connected, so it is not one of the non--planar graphs described in Remark~\ref{rmknonpl}.
As $G - a,v$ is planar, $G$ is $2$--apex.

If $\Gab$ is $K_2 \sqcup H$ where $H$ is graph viii of 
Figure~\ref{figG811}, again, $a$ is adjacent to a degree three vertex of
$H$. If that vertex is one of the six $v_i$ or $w_i$ vertices, the argument proceeds as above. So assume instead $a$ is adjacent to the seventh
degree three vertex. In this case $G - v_1$ is planar so $G$ is $1$--apex, hence $2$--apex.

If $\Gab$ is the union of the right graph of Figure~\ref{figG710}, call it $H$, with a tree $T$ of two edges, we again conclude that if $a$ is adjacent to $v$, a degree three vertex, of $H$ then $G - a,v$ is planar whence $G$ is $2$--apex. The only other way to produce a degree four vertex
for $G$ is if $a$ and $b$ are both adjacent to all three vertices of $T$. However, in this case we find that the subgraph induced by $a$, $b$, and the vertices of the tree is planar so that $G$ is $1$--apex and, therefore, also $2$--apex.

Similar arguments apply when $\Gab$ is the union of $K_{3,3}$ and the tree $P_3$, the path of 
three edges: either $a$ is adjacent to a vertex $v$ of $K_{3,3}$,
which means that $G - a,v$ is planar, or else the graph induced by $a$, $b$
and $P_3$ is planar so that $G$ is, in fact, $1$--apex, hence $2$--apex.
As for $K_{3,3} \sqcup S_3$, where $S_3$ is the star of three edges, again $G - a,v$ is planar where $v$ is the vertex of $K_{3,3}$ adjacent to $a$ if
there is such and otherwise $v$ is an arbitrary vertex of $K_{3,3}$.
This completes the argument when $\maxd{G} = 5$.

Finally, suppose $\maxd{G} = 4$. Then there are four degree four
vertices with the remaining vertices of degree three. If there are 
non--adjacent degree four vertices $a$ and $b$, then $\| \Gab \| = 12$ and the analysis is much as the one just completed in the $\maxd{G} = 5$ case.
That is, we can assume
$\Gab$ is one of the fifteen graphs described in Remark~\ref{rmknonpl} 
with the additional condition that $\maxd{\Gab} \leq 4$.

So, to complete the proof, let's assume the four degree four vertices, call them $a$, $b$, $c$, and $d$, are mutually adjacent.
Then $c$ and $d$ become two adjacent degree
two vertices in  $\Gab$. Smoothing these we
arrive at $G'$ where $|G'| = 8$ and $\| G' \| = 11$. We can assume that 
$G'$ is non--planar (otherwise $\Gab$ is planar and $G$ is $2$--apex) so that it 
is one of the graphs of Figure~\ref{figG811}, $K_{3,3} \sqcup K_2$ with an
edge doubled, or else the union of one of the graphs of Figure~\ref{figG710} 
with $C_1$, a loop on one vertex. In addition, $\maxd{G'} = 3$, which leaves six possibilities: graph viii, ix, x, or xi of Figure~\ref{figG811},
$K_{3,3} \sqcup C_2$, where $C_2$ is the cycle on two vertices, or else
the union of $C_1$ and the graph at the right of Figure~\ref{figG710}.
We'll consider these in turn.

If $G'$ is graph viii, ix, x, or xi of Figure~\ref{figG811}, let $xy$ be 
the edge of $G'$ that contained $c$ and $d$ before smoothing. That is,
$x$ and $y$ are the vertices in $\Gab$ such that 
$xc$, $cd$, and $dy$ is a path. Then $G - x,y$ is planar and $G$ is $2$--apex.

If $G'$ is $K_{3,3} \sqcup C_2$, then $\Gab$ is $K_{3,3} \sqcup C_4$ with
$c$ and $d$ two of the vertices in the $4$--cycle $C_4$. Then $G-a,v_1$ is planar where $v_1$ is a vertex of $K_{3,3}$. Finally, if $G'$ is
the union of $C_1$ and the right graph of Figure~\ref{figG710}, call it $R$, then $\Gab$ is $C_3 \sqcup R$ where $c$ and $d$ are two of the vertices
in the $3$--cycle $C_3$. It follows that $G - v_1$ is planar so that 
$G$ is $1$--apex, hence $2$--apex.

This completes the case where $\maxd{G} = 4$, and with it the proof for 
$\|G \| = 20$. As usual, since all graphs with $\| G \| = 20$ are $2$--apex,
the same is true for graphs with $\|G \| \leq 20$.
\end{proof}

\subsection{Thirteen or more vertices}

In this subsection, we complete the proof of Theorem~\ref{thmain} by examining graphs with 13 or more vertices.

\begin{prop} A graph $G$ with $|G| \geq 13$ and $\|G \| \leq 20$ is $2$--apex.
\end{prop}

\begin{proof}
Suppose $|G| = 13$ and $\|G \| = 20$. By Lemma~\ref{lemd12}, we can assume $\mind{G} \geq 3$ so that $G$ has a single vertex $a$ of degree
four with all other vertices of degree three. Let $b$ be a vertex that
is not adjacent to $a$ so that $\| \Gab \| = 13$. 
Assume $G$ is not $2$--apex.
Then $\Gab$ is non--planar. Now, $\maxd{\Gab} = 3$, so $\Gab$ has no $K_5$ component. By Remark~\ref{rmknonpl}, $\Gab$ has exactly one tree component $T$, with the rest of the graph $G' = \Gab \setminus T$
having a $K_{3,3}$ minor. As $\mind{\Gab} \geq 1$, there are no isolated degree zero vertices, so $2 \leq |T| \leq 5$ and we have four cases. 

If $|T| = 2$, then $T$ is $K_2$ and $G' = \Gab \setminus T$ is a non-planar graph on nine vertices with 12 edges. As $\maxd{G'} = 3$ and $\mind{G'}  \geq 1$, $G'$ has a vertex of degree two. By smoothing that vertex,
we have either  a multigraph obtained by doubling an edge of the graph $K_{3,3} \sqcup K_2$ or else one of the graphs of Figure~\ref{figG811}.
Moreover, as $\maxd{G'} = 3$, of the graphs in the figure, only viii, ix, x, and xi are possibilities. 

Suppose then that, after smoothing and simplifying, 
$G'$ is $K_{3,3} \sqcup K_2$. Then, as
$\maxd{G'} = 3$, the doubled edge is that of the $K_2$ and
$G' = K_{3,3} \sqcup C_3$, where $C_3$ denotes the cycle on three vertices. Thus, $\Gab = K_{3,3} \sqcup C_3 \sqcup K_2$. Let $c$ be one of the
vertices in the $K_{3,3}$ component. Then $G - a,c$ is planar and $G$ is $2$--apex.

If, after smoothing a degree two vertex, $G'$ becomes graph viii, ix, x, or xi of Figure~\ref{figG811}, then $G - a, v_1$ is planar 
and $G$ is $2$--apex.

Next suppose $|T| = 3$. As, $|G'| = 8$, $\|G'\| = 11$, and $\maxd{G'} = 3$, we conclude that $G'$ is graph viii, ix, x, or xi
of Figure~\ref{figG811}. Whichever it is, $G - a, v_1$ will be a planar
subgraph of $G$ so that $G$ is $2$--apex. 

Similarly, if $|T| = 4$, 
then $|G'| = 7$, $\|G'\| = 10$. As $\maxd{G'} = 3$, we conclude that $G'$ is the graph to 
the right of Figure~\ref{figG710}. Then $G - a,v_1$ is planar and 
$G$ is $2$--apex.  

Finally, if $|T|=5$, then
$|G'| = 6$ and $\|G'\| = 9$ so that $G'$ is $K_{3,3}$. 
Again, $G - a,v_1$ is planar and $G$ is $2$--apex.  

We have shown that a graph with $|G| = 13$ and $\|G\| = 20$ is $2$--apex. It follows that the same is true for graphs having $|G| = 13$ and $\|G \| \leq 20$.

Now, suppose $|G| \geq 14$ and $\|G\| = 20$.
If $\mind{G} \geq 3$, then the degree sum is at least $3 \times 14= 42> 40$, a contradiction. So, we may assume $\mind{G} < 3$ which implies
$G$ is $2$--apex by Lemma~\ref{lemd12}.
It follows that any graph of 14 or more vertices with fewer than 20 edges is also $2$--apex.
\end{proof}

\section{Graphs on twenty-one edges}

In this section we prove Propositions~\ref{prop821} (in the first subsection) and Propositions~\ref{prop921$2$--apex}
and \ref{prop921IK} (in the second subsection).

\subsection{Eight or fewer vertices}
In this subsection we prove Proposition~\ref{prop821}, a non-IK graph
of eight or fewer vertices is $2$--apex.  This implies that for these graphs $2$--apex is equivalent to not IK and the classification of $2$--apex graphs on eight or fewer vertices follows from the IK classification due to \cite{BBFFHL} and \cite{CMOPRW}. 

\setcounter{section}{1}
\setcounter{theorem}{3}

\begin{prop} 
Every non IK graph on eight or fewer vertices is $2$--apex.
\end{prop}

\setcounter{section}{3}
\setcounter{theorem}{0}

\begin{proof}
By Proposition~\ref{prop820}, the theorem holds if $\|G\| \leq  20$, so we may assume that $\|G\| \geq 21$. 
The only graph with 21 edges and fewer than eight vertices is $K_7$, which is IK. So we may assume $|G| = 8$.

\begin{figure}[ht]
\begin{center}
\includegraphics[scale=0.5]{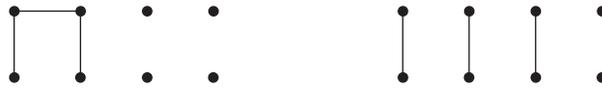}
\caption{Complements of the non IK graphs $G_1$ and $G_2$.\label{fig8G1G2}}
\end{center}
\end{figure}

Knotting of graphs on 
eight vertices was classified independently by~\cite{CMOPRW} and \cite{BBFFHL}. Using the classification, the non IK graphs with 21 or more edges
are all subgraphs of two graphs on 25 edges, $G_1$ and $G_2$, whose 
complements appear in Figure~\ref{fig8G1G2}. Each of these two graphs has at least two vertices of degree seven and, for both graphs, deleting two such vertices leaves a planar subgraph of $K_6$. Thus, both $G_1$ and $G_2$ are $2$--apex and the same is true of any subgraph of $G_1$ and $G_2$.
\end{proof}

\subsection{Nine vertices}
In this subsection we prove Propositions~\ref{prop921$2$--apex} and \ref{prop921IK}, which classify the graphs of nine vertices and at most 21 
edges with respect to $2$--apex and IK. 

We begin with Proposition~\ref{prop921$2$--apex}: among these graphs,
all but $E_9$ (see Figure~\ref{figE9}) and 
four graphs derived from $K_7$ by triangle--Y moves
($K_7 \sqcup K_1 \sqcup K_1$, $H_8 \sqcup K_1$, $F_9$, and $H_9$,
see~\cite{KS}) are $2$--apex. We first present four lemmas that
show this is the case when there is a subgraph $\Gab$ of the form shown
in Figure~\ref{figG711}. The first lemma shows that we can assume $\mind{\Gab} \geq 2$. The next three treat the five graphs (iii, iv, v, vi and viii) of Figure~\ref{figG711} that meet this condition.

\begin{lemma} \label{lem9.1}%
Let $G$ be a graph with $|G| = 9$, $\|G \| = 21$ and $\mind{G} \geq 3$. Suppose that, 
for each pair of vertices $a'$ and $b'$, $\| G - a',b' \| \geq 11$ with 
equality for at least one pair $a$, $b$. Then, $a$ and $b$ can be chosen so that one of the following two holds.
\begin{itemize}
\item The vertices $a$ and $b$ have degrees six and five, respectively, $G$ has one of the following degree sequences: $\{6,5,5,5,5,5,5,3,3\}$, 
$\{6,5,5,5,5,5,4,4,3\}$, or $\{6,5,5,5,5,4,4,4,4\}$, and $a$ is 
adjacent to each degree five vertex (including $b$).
\item The vertices $a$ and $b$ both have degree five, $G$ has one
of the following degree sequences:  $\{5,5,5,5,5,5,5,4,3\}$ or $\{5,5,5,5,5,5,4,4,4\}$, and $a$ and $b$ are not neighbours.
\end{itemize}
Moreover, $a$ and $b$ can be chosen so that $\mind{\Gab} \geq 2$.
\end{lemma}

\begin{proof}
We can assume $\maxd{G} = d(a) \geq d(b)$. Since $\| G - a',b' \| \geq 11$ 
for every pair of vertices $a'$, $b'$, we must have $d(a) = 6$ or $d(a) = 5$. 

If $d(a) = 6$, the condition $\| G - a', b' \| \geq 11$ implies
that there is exactly one degree six vertex, $a$, and every degree five vertex is adjacent to $a$. As $\| G \| = 21$, the degree sum is 42 and, therefore, there are only three possibilities for the degree sequence. In particular, there is always a vertex of degree five $b$ which is adjacent
to $a$ so that $\| \Gab \| = 11$. 

Similarly, if $d(a) = 5$, then the 
condition $\| G \| = 21$ leaves two possible degree sequences. There must be two degree five vertices $a$ and $b$ that are not adjacent
so that $\| \Gab \| = 11$. This is clear for the degree sequence with seven degree five vertices. In the
case of six degree five vertices, if they were all mutually adjacent, they 
would constitute a $K_6$ component of 15 edges. The other component has
three vertices and, at most, three edges. In total, $G$  would have
at most 18 edges, a contradiction.

Finally, we argue that it is always possible to choose $a$ and $b$ so that $\mind{\Gab} \geq 2$. Indeed, this is obvious
when $\mind{G} \geq 4$ as deleting $a$ and $b$ can reduce the degree of the other vertices by at most two. For the sequence $\{6,5,5,5,5,5,5,3,3\}$, the degree six vertex $a$ is adjacent to each degree five vertex and is, therefore, not adjacent to either of the degree three vertices. Hence in 
$\Gab$ these degree three vertices have degree at least two and $\mind{\Gab} \geq 2$. For $\{6,5,5,5,5,5,4,4,3\}$, the degree three vertex is
adjacent to at most three of the degree five vertices. By choosing $b$ as one of the other degree five vertices, we will have $\mind{\Gab} \geq 2$.
Similarly, for $\{5,5,5,5,5,5,5,4,3\}$, the degree three vertex is adjacent to at most three of the degree five vertices, call them $v_1$ $v_2$, and $v_3$. We can find degree five vertices $a$ and $b$ that are not adjacent and not both neighbours of the degree three vertex (so that $\mind{\Gab} \geq 2$). For,
if not, then the remaining four degree five vertices, $v_4$, $v_5$, $v_6$, and $v_7$ are all mutually adjacent and also all adjacent to $v_1$, $v_2$, and $v_3$. But this is not possible, e.g., $d(v_4) = 5$, so it
cannot have all of the other degree five vertices as neighbours.
\end{proof}

\begin{lemma} \label{lem9iii}%
Let $G$ be a graph with $|G| = 9$, $\|G \| = 21$, and $\mind{G} \geq 3$. Suppose that there
are vertices $a$ and $b$ such that $\Gab$ is graph iii of Figure~\ref{figG711} and for any pair of vertices $a'$ and $b'$, 
$\| G-a',b' \| \geq 11$. If $G$ is not $2$--apex, then $G$ is $H_9$.
\end{lemma}

\begin{proof}
By Lemma~\ref{lem9.1}, $d(a) = 6$ or $5$. Also, since $u$ has degree three or more in $G$, at least one of $a$ and $b$ is adjacent to $u$.

Assuming $G$ is not $2$--apex, by Lemma~\ref{lemgenK33},
$\{w_1, w_2, w_3\} \subset N(a) \cap N(b)$. 
Then $G - u,w_1$ is planar (and $G$ is $2$--apex) in the case 
$d(a) = 5$. 

So, we can assume $d(a) = 6$ and we are in the first case of 
Lemma~\ref{lem9.1}. As above $\{w_1, w_2, w_3 \} \subset 
N(a) \cap N(b)$. Then, since $G - a, w_1$ is non--planar, we deduce
that $N(b) = \{a, w_1, w_2, w_3, u \}$. Finally, since $G - w_1,w_2$ is non--planar, $v_1$ and $v_2$ are also neighbours of $a$, i.e., 
$N(a) = \{b,v_1,v_2,w_1,w_2,w_3 \}$. But these choices of $N(b)$ and $N(a)$ result in the graph $H_9$. So, if $G$ is not $2$--apex, it is $H_9$.
\end{proof}

\begin{lemma} \label{lem9ivtovi}%
Let $G$ be a graph with $|G| = 9$, $\|G \| = 21$, and $\mind{G} \geq 3$. 
Suppose that there
are vertices $a$ and $b$ such that $\Gab$ is graph iv, v, or vi of Figure~\ref{figG711} and for any pair of vertices $a'$ and $b'$, 
$\| G-a',b' \| \geq 11$. Then $G$ is $2$--apex.
\end{lemma}

\begin{proof}
By Lemma~\ref{lem9.1}, $d(a) = 6$ or $5$. 
If $d(a) = 5$, note that $G - a,b,v_2,w_2$ is a cycle. By placing $a$ inside the cycle and $b$ outside, $G - v_3, w_3$ is planar and $G$ is $2$--apex.
So, we may assume $d(a) =6$ and we are in the first case of Lemma~\ref{lem9.1}.

Assume $G$ is not $2$--apex and apply Lemma~\ref{lemgenK33} to $(\Gab; v_2)$, for which $W_1 = \{u,w_1\}$ and $W_i = \{w_i\}$, $i=2,3$.
If $G$ is graph iv or v, then $G - w_2, w_3$ is planar and $G$ is $2$--apex. So, let $G$ be graph vi.
Then, since $G - w_2, w_3$ is non--planar, either $v_1$ or $v_2$, say $v_1$, is a neighbour of $b$. But, then the degree of $v_1$ in $G$ is at least five. 
We deduce that $av_1 \in E(G)$, for otherwise, $d(v_1) = 5$ and, by Lemma~\ref{lem9.1}, $v_1$ is adjacent to $a$, a contradiction. However,
as $a$ is adjacent to $v_1$, $d(v_1) = 6$ which again contradicts Lemma~\ref{lem9.1} as $a$ is the unique vertex of degree six. The contradiction 
shows that $G$ is $2$--apex.
\end{proof}

\begin{figure}[ht]
\begin{center}
\includegraphics[scale=0.5]{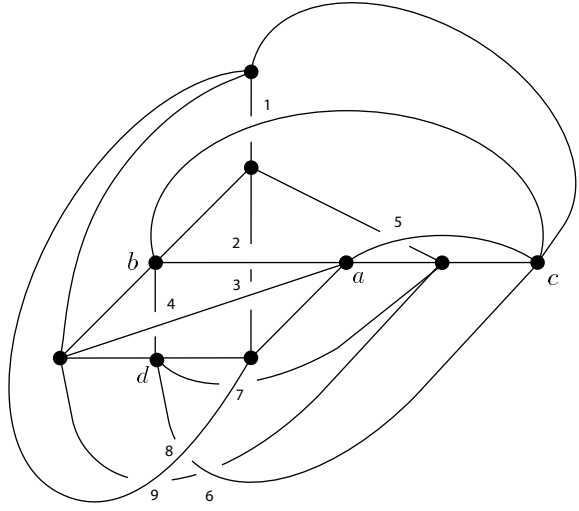}
\caption{An unknotted embedding of the graph $E_9$.\label{figE9}}
\end{center}
\end{figure}

\begin{lemma} \label{lem9viii}%
Let $G$ be a graph with $|G| = 9$, $\|G \| = 21$, and $\mind{G} \geq 3$. 
Suppose that there
are vertices $a$ and $b$ such that $\Gab$ is graph viii of Figure~\ref{figG711} and for any pair of vertices $a'$ and $b'$, 
$\| G-a',b' \| \geq 11$. If $G$ is not $2$--apex, then $G$ is $E_9$.
\end{lemma}

\begin{proof}
By Lemma~\ref{lem9.1}, $d(a) = 6$ or $5$. Also, since $u$ has degree three or more in $G$, at least one of $a$ and $b$ is adjacent to $u$.

Assume $G$ is not $2$--apex and apply Lemma~\ref{lemgenK33} using $W_1 = \{u,w_1\}$ and $W_i = \{w_i\}$, $i = 2,3$,
to see that  $\{w_2,w_3 \} \subset N(a) \cap N(b)$ and that $N(a)$ and $N(b)$ both intersect $W_1$.
Similarly $(\Gab; w_1)$ shows $\{v_2,v_3\} \subset N(a) \cap N(b)$ and both $a$ and $b$ have a neighbour in $V_1 = \{u,v_1\}$.
If $d(a) = 6$, then $|N(b) \cap (V(G) \setminus \{a,b\})| = 4 $, which contradicts what we already know about 
$N(b)$. So, it must be that $d(a) = 5$, from which it follows that $N(a) = N(b) = \{u, v_2, v_3, w_2, w_3 \}$ and that $G = E_9$.
\end{proof}

Having treated graphs containing an induced subgraph as in Figure~\ref{figG711}, we are ready to 
prove Proposition~\ref{prop921$2$--apex}.

\setcounter{section}{1}
\setcounter{theorem}{5}

\begin{prop} 
 Let $G$ be a graph with $|G| = 9$ and $\|G \| \leq 21$. If $G$ is not $2$--apex, then $G$ is either $E_9$ or else one of the following IK graphs: $K_7 \sqcup K_1 \sqcup K_1$, $H_8 \sqcup K_1$, $F_9$, or $H_9$.
\end{prop}

\setcounter{section}{3}
\setcounter{theorem}{6}

\begin{proof}
By Theorem~\ref{thmain}, we can assume $\|G \| = 21$.

As in the proof of Lemma~\ref{lemd12}, $\mind{G} \geq 3$ unless $G$ has a vertex of degree lower than three whose deletion 
(or smoothing in the case of a degree two vertex) results in a graph that is not $2$--apex. As all graphs of 20 edges are $2$--apex, this is possible 
only in the case that $G$ has a degree zero vertex; deleting that vertex must result in a graph on eight vertices with 21 edges
that is not $2$--apex. 
By Proposition~\ref{prop821} such a graph is IK and, by the classification of knotting of
eight vertex graphs, we conclude that $G$ is either the union of $K_7$
with two degree zero vertices, $K_7 \sqcup K_1 \sqcup K_1$, or else
$G$ is $H_8 \sqcup K_1$, where $H_8$ is the graph obtained by a single
triangle--Y move on $K_7$ (see~\cite{KS}).

In other words, so long as $G \neq K_7 \sqcup K_1 \sqcup K_1$ and 
$G \neq H_8 \sqcup K_1$, we can assume $\mind{G} \geq 3$. 
Also, $5 \leq \maxd{G} \leq 8$. Now, 
if $\maxd{G} = 5$, there are at least six degree five vertices
and, therefore, there must be a pair of non--adjacent degree five
vertices. Thus, whatever the maximum degree $\maxd{G}$, by appropriate choice of vertices 
$a$ and $b$, we may assume $\Gab$ has at most 11 edges. 

If $\maxd{G} = 8$, then $G$ is $2$--apex. Indeed, if $a$ has degree
eight, then, as $\|G \| =21$, there is a vertex $b$ with $d(b) \geq 5$.
This means $\Gab$ has at most nine edges and,  
by Lemma~\ref{lemnonpl}, is planar. So we'll assume $\Gab$ has at most
11 edges and that $7 \geq d(a) \geq d(b).$
Assume $G$ is not $2$--apex; then $\Gab$ is non--planar. By Remark~\ref{rmknonpl}, 
$\Gab$ is one of the two graphs in Figure~\ref{figG710} or one of the
nine in Figure~\ref{figG711}.

Suppose first that $\Gab$ is the graph at left in  Figure~\ref{figG710}. 
Since $\| \Gab \| = 10$, 
we can assume that $d(a) = 7$ or $6$ and $d(b) \leq 6$.
As $u$ has degree three in $G$, both $a$ and $b$ are adjacent to $u$. 
By Lemma~\ref{lemgenK33}, $\{ w_1, w_2, w_3 \} \subset N(b)$. Without loss of
generality, we can assume $aw_1 \in E(G)$. Then $G - a, w_1$ is planar and
$G$ is $2$--apex.

If $\Gab$ is the graph at right in Figure~\ref{figG710}, then,
as $u$ has degree three or more in $G$, 
at least one of $a$ and $b$ is a neighbour of $u$.
Again, $\|\Gab\|=10$ so $d(a) = 7$ or $6$ and $d(b) \leq 6$.
Applying Lemma~\ref{lemgenK33} with $W_1 = \{u, w_1\}$ and $W_i = \{w_i\}$, $i = 2,3$,
we see that $N(a)$ and $N(b)$ each include at least one vertex from each $W_i$.
Similarly, $(\Gab;w_1)$ shows $N(a)$ and $N(b)$ each include at least one vertex from each of $V_1 = \{u, v_1\}$ and $V_i = \{v_i\}$. $i = 2,3$.
In particular, we conclude that $\{v_2,v_3,w_2,w_3\} \subset N(a) \cap N(b)$.
Then $G - w_2, w_3$ is a non--planar graph on seven vertices and eleven edges, i.e., one of the graphs 
in Figure~\ref{figG711}. 

In particular, if $d(a) = 7$, then the degree
of $a$ in $G - w_2, w_3$ is five. The only graph of 
Figure~\ref{figG711} with
a degree five vertex is i. However, in that graph, the degree five vertex
is adjacent to a degree one vertex which is not a possibility for $a$. 
So, we conclude $G - w_2, w_3$ is planar and $G$ is $2$--apex if $d(a) = 7$.

Thus, we can assume $d(a) = 6$ and $d(b) = 5$ or $6$. If $d(b) = 5$, then the discussion above shows that $N(b) = \{u,v_2,v_3,w_2,w_3\}$ and
$G - v_3, w_3$ is planar (so that $G$ is $2$--apex). If $d(b) = 6$, then
$ab \in E(G)$ which implies $N(a) \cap N(b) =  \{u,v_2,v_3,w_2,w_3\}$ 
and $G = F_9$.

We can now assume that there is a pair of vertices $a$ and $b$ such
that $\| \Gab \| = 11$ and $\Gab$ is one of the nine graphs in Figure~\ref{figG711}. Moreover, we can also posit that for any pair of 
vertices $a'$, $b'$, $\|G - a',b' \| \geq 11$, for otherwise the subgraph 
has ten vertices (in order to ensure it is non--planar, see Lemma~\ref{lemnonpl}), 
which is the case we just treated above. Lemma~\ref{lem9.1} describes 
the possible degrees for such a graph. In particular, $\mind{\Gab} \geq 2$ 
and since $G$ is not $2$--apex, $\Gab$ must be graph iii, iv, v, vi, or viii in Figure~\ref{figG711}. 
Lemmas~\ref{lem9iii} through \ref{lem9viii} show that in those cases, if $G$ is not $2$--apex, 
then $G$ is $E_9$ or $H_9$. This completes the proof. 
\end{proof}

Finally, we prove Proposition~\ref{prop921IK}.

\setcounter{section}{1}
\setcounter{theorem}{6}

\begin{prop} Let $G$ be a graph with $|G| = 9$ and $\|G \| \leq 21$.
Then $G$ is IK iff it is $K_7 \sqcup K_1 \sqcup K_1$, $H_8 \sqcup K_1$, $F_9$, or $H_9$.
\end{prop}

\setcounter{section}{3}
\setcounter{theorem}{6}

\begin{proof} 
It follows from \cite{KS} that, if $G$ is one of the four listed
graphs, then it is IK.
 
By Proposition~\ref{prop920}, $G$ is $2$--apex and, therefore, not IK when $\|G \| \leq 20$. In case $\| G \| = 21$, Proposition~\ref{prop921$2$--apex} shows $G$ is $2$--apex and not IK unless $G$ is one of the four listed graphs or $E_9$.
However, Figure~\ref{figE9} is an unknotted embedding of $E_9$. So, if $G$ is IK,
it must be one of the four listed graphs.
\end{proof}

\begin{rmk}
It is straightforward to verify that Figure~\ref{figE9} is an unknotted embedding. For example, here's a strategy for making such a verification. Number the crossings as shown.
It is easy to check that there are 16 possible crossing combinations for a cycle in this graph:
1236, 134, 1457, 1459, 1479, 16789, 234689, 23479, 236789, 25689, 2578, 345689, 3457, 3578, 36789, and 4578.
That is, any cycle in the graph will have crossings that are a subset of one of those 16 sets. For example, a cycle that includes crossings
1, 2, 3, and 6 must include the edges $ab$ and $bc$ and therefore, cannot have the edge $bd$ required for crossing 4. Indeed, 
a cycle that includes 1, 2, 3, and 6 can have none of the other five crossings. To show that there are no knots, consider cycles that correspond to each subset of the 16 sets and check that each such cycle (if such exists) is not knotted in the embedding of Figure~\ref{figE9}.
\end{rmk}

\section*{Acknowledgements}

We thank Ramin Naimi for encouragement and for sharing an unknotted embedding of the graph $E_9$ and
Joel Foisy  for helpful conversations. This study was inspired by the Master's thesis of Chris Morris.

\end{document}